\documentclass[12pt, reqno]{amsart}
\usepackage{ amsmath,amsthm, amscd, amsfonts, amssymb, graphicx, color}
\usepackage[bookmarksnumbered, colorlinks, plainpages]{hyperref}
\newtheorem{thm}{Theorem}[section]
\newtheorem{cor}[thm]{Corollary}
\newtheorem{lem}[thm]{Lemma}

\newtheorem{exam}[thm]{Example}
\numberwithin{equation}{section}

\begin{document}

\title[the sum of two idempotents and a nilpotent]{Rings over which every matrix is the sum of two idempotents and a nilpotent}

\author{Huanyin Chen}
\author{Marjan Sheibani}
\address{
Department of Mathematics\\ Hangzhou Normal University\\ Hang -zhou, China}
\email{<huanyinchen@aliyun.com>}
\address{
Faculty of Mathematics\\Statistics and Computer Science\\  Semnan University\\ Semnan, Iran}
\email{<m.sheibani1@gmail.com>}


\begin{abstract}
A ring $R$ is (strongly) 2-nil-clean if every element in $R$ is the sum of two idempotents and a nilpotent (that commute).
Fundamental properties of such rings are discussed. Let $R$ be a 2-primal ring. If $R$ is strongly 2-nil-clean, we show that
$M_n(R)$ is 2-nil-clean for all $n\in {\Bbb N}$. We also prove that the matrix ring is 2-nil-clean for a strongly 2-nil-clean ring of
bounded index. These provide many classes of rings over which
every matrix is the sum of two idempotents and a nilpotent.
\end{abstract}

\maketitle

\section{Introduction}

Throughout, all rings are associative with an identity. A ring is called ¡°(strongly) nil-clean¡± if every element can be written as the sum
of an idempotent and a nilpotent (that commute). A ring $R$ is weakly nil-clean provided that every element in $R$ is the sum or difference of a nilpotent element and an idempotent. Such rings have been the object of much
investigation over the last decade, as they are related to the well-studied clean
rings of Nicholson. Though nil and weakly clean rings are popular, the conditions a bit
restrictive (for example, there are even fields which are not weakly nil clean). The subjects of nil-clean and weakly nil-clean rings are interested for so many mathematicians, e.g., ~\cite{BD,CH,DW,D,K,KWZ,KW} and ~\cite{ST}. In the current paper, we seek to remedy this by looking at an interesting
generalization of nil and weakly nil cleanness, which they call ¡°2-nil-clean¡±. That is, a ring $R$ is (strongly) 2-nil-clean provided that every element in $R$ is the sum of two idempotents and a nilpotent (that commute). This new class enjoys many interesting
properties and examples (for example, all tripotent rings are 2-nil-clean). We shall investigate when a matrix ring is 2-nil-clean, i.e.,
when every matrix over a ring can be written as the sum of two idempotents and a nilpotent. A ring $R$ is 2-primal if its prime radical coincides with the set of nilpotent elements of the ring. Examples of 2-primal rings include commutative rings and reduced rings. Let $R$ be a 2-primal ring. If $R$ is strongly 2-nil-clean, we show that $M_n(R)$ is 2-nil-clean for all $n\in {\Bbb N}$. A ring $R$ is of bounded index if there is a
positive integer $n$ such that $a^n=0$ for each nilpotent element $a$ of
$R$. We also prove that the matrix ring is 2-nil-clean for a strongly 2-nil-clean ring of
bounded index. These provide many classes of rings over which
every matrix is the sum of two idempotents and a nilpotent.

We use $N(R)$ to denote the set of all nilpotent elements in $R$ and $J(R)$ the Jacobson radical of $R$.
${\Bbb N}$ stands for the set of all natural numbers.

\section{Examples and Subclasses}

\vskip4mm The aim of this section is to construct examples of 2-nil-clean rings and investigate certain subclass of such rings. We begin with

\begin{exam} The class of 2-nil-clean rings contains many familiar examples.\end{exam}
\begin{enumerate}
\item [(1)] Every weakly nil-clean ring is 2-nil-clean, e.g., strongly nil-clean rings, nil-clean rings, Boolean rings, weakly Boolean rings.
\item [(2)] ${\Bbb Z}_3\times {\Bbb Z}_3$ is 2-nil-clean, while it is not weakly nil-clean.
\item [(3)] A local ring $R$ is 2-nil-clean if and only if $R/J(R)\cong {\Bbb Z}_2$ or ${\Bbb Z}_3$, and $J(R)$ is nil.
\end{enumerate}

We also provide some examples illustrating which ring-theoretic extensions of 2-nil-clean rings produce 2-nil-clean rings.

\begin{exam} \end{exam}
\begin{enumerate}
\item [(1)] Any quotient of a 2-nil-clean ring is 2-nil-clean.
\item [(2)] Any finite product of 2-nil-clean rings is 2-nil-clean. But $R={\Bbb Z}_2\times {\Bbb Z}_4\times {\Bbb Z}_8\times $ is an infinite product of 2-nil-clean rings, which is not 2-nil-clean. Here, the element $(0,2,2,2,\cdots )\in R$ can not written as the sum of two idempotents and a nilpotent element.
\item [(3)] The triangular matrix ring $T_n(R)$ over a 2-nil-clean ring $R$ is 2-nil-clean.
\item [(4)] The quotient ring $R[[x]]/(x^n) (n\geq 1)$ of a 2-nil-clean ring $R$ is 2-nil-clean.
\end{enumerate}

\begin{thm} Let $I$ be a nil ideal of the ring $R$. Then $R$ is 2-nil-clean if and only if the quotient ring $R/I$ is 2-nil-clean.\end{thm}
\begin{proof} $\Longrightarrow$ It is obtained from Example 2.2 (1).

$\Longleftarrow $ Let $a\in R$, there exist two idempotents $\overline{e}, \overline{f}\in R/I$ and a nilpotent $\overline{w}\in R/I$ such that $\overline{a}=\overline{e}+\overline{f}+\overline{w}$. As idempotents and nilpotents lift modulo nil ideal, we can assume that $e, f$ are idempotents in $R$ and $w$ is a nilpotent in $R$. Then $a=e+f+w+r$ for some $r\in I$. Since $w\in N(R)$, we may assume that $w^k=0$ for some $k\in {\Bbb N}$, this implies that $(w+r)^k\in I$ and so $w+r\in N(R)$. This completes the proof.\end{proof}

We use $P(R)$ to denote the prime radical of a ring $R$. That is, $P(R)=\bigcap \{ P~|~P~\mbox{is a prime ideal of}~R\}$. We have

\begin{cor} A ring $R$ is 2-nil-clean if and only if the quotient ring $R/P(R)$ is 2-nil-clean.\end{cor}
\begin{proof} As $P(R)$ is a nil ideal of $R$, the result follows from Theorem 2.3.\end{proof}

\begin{cor} Let
$R$ be a ring. Then the following are equivalent:\end{cor}
\begin{enumerate}
\item [(1)] {\it $R$ is 2-nil-clean.}
\vspace{-.5mm}
\item [(2)] {\it $T_n(R)$ is 2-nil-clean for all $n\in {\Bbb N}$.}
\vspace{-.5mm}
\item [(3)] {\it
$T_n(R)$ is 2-nil-clean for some $n\in {\Bbb N}$.} \vspace{-.5mm}
\end{enumerate} \begin{proof}  $(1)\Rightarrow (2)$ Let $I=\{ \left(
\begin{array}{cccc}
0&a_{12}&\cdots& a_{1n}\\
0&0&\cdots&a_{2n}\\
\vdots&\vdots&\ddots&\vdots\\
0&0&\cdots&0
\end{array}
\right)~|~a_{ij}\in R\}$. Then $I$ is an ideal of $T_n(R)$. Clearly, $T_n(R)/I\cong R\times R\times \cdots \times R$. In light of Example 2.2 and Theorem 2.3, we show that $T_n(R)$ is 2-nil-clean.

$(2)\Rightarrow (3)$ is trivial.

$(3)\Rightarrow (1)$ Straightforward.\end{proof}

A ring $R$ is tripotent if $a^3=a$ for all $a\in R$. We have

\begin{lem} Let $R$ be a ring. Then the following are equivalent:\end{lem}
\vspace{-.5mm}
\begin{enumerate}
\item [(1)]{\it $R$ is tripotent.}
\vspace{-.5mm}
\item [(2)]{\it $R$ is a commutative ring in which every element is the sum of two idempotents.}
\vspace{-.5mm}
\item [(3)]{\it $R$ is the product of fields isomorphic to ${\Bbb Z}_2$ or ${\Bbb Z}_3$.}
\end{enumerate}
\begin{proof} $(1)\Longleftrightarrow (2)$  This is obvious, by ~\cite[Theorem 1]{HT}.

$(1)\Rightarrow (3)$ Birkhoff¡¯s Theorem, $R$ is isomorphic to a subdirect product of
subdirectly irreducible rings $R_i$. Thus, $R_i$ satisfies the identity $x^3=x$. In view of ~\cite[Theorem 1]{HT}, $R_i$ is commutative.
But $R_i$ has no central idempotents except for $0$ and $1$. Thus,
$x^2=0$ or $x^2=1$. Hence, $x=x^3=0$ or $x^2=1$. If $x\neq 0,1$, then $(x-1)^2=1$, and so $x(x-2)=0$. This implies that $x=2$.
Thus, $R_i\cong {\Bbb Z}_2$ or ${\Bbb Z}_3$, as desired.

$(3)\Rightarrow (1)$ $R$ is the product of fields isomorphic to ${\Bbb Z}_2$ or ${\Bbb Z}_3$. As ${\Bbb Z}_2$ and
${\Bbb Z}_3$ satisfy the identity $x^3=x$. This completes the proof.\end{proof}

Clearly, strongly 2-nil-clean rings form a subclass of 2-nil-clean rings. For further use, we now consider strongly 2-nil-clean rings. We record the following.

\begin{lem} A ring $R$ is strongly 2-nil-clean if and only if\end{lem}
\begin{enumerate}
\item [(1)] {\it $J(R)$ is nil;}
\vspace{-.5mm}
\item [(2)] {\it $R/J(R)$ is tripotent.}
\end{enumerate}
\begin{proof}
$\Longrightarrow$ Let $a\in R$. Then we can find two idempotents $e,f\in R$ and a nilpotent $w\in R$ such that $a+1=e+f+w$ where $e,f$ and $w$ commute.
Hence, $a=e-(1-f)+w$. Clearly, $\big(e-(1-f)\big)^3=e-(1-f)$, we see that
$a^3-a\in N(R)$. It follows by~\cite[Theorem A.1]{HTY} that $N(R)$ forms an ideal of $R$. Hence, $N(R)\subseteq J(R)$. This shows that every element in $R/J(R)$ is the sum of two idempotents that commute. In view of Lemma 2.6, $R/J(R)$ is tripotent.
Let $x\in J(R)$. Then $x^3-x\in N(R)$ by the preceding discussion, Hence, $w:=x(1-x^2)\in N(R)$. This implies that $x=w(1-x^2)^{-1}\in N(R)$; hence,
$J(R)$ is nil, as desired.

$\Longleftarrow$ By hypothesis, $\overline{2}=\overline{2^3}$ in $R/J(R)$. Hence, $6\in J(R)$ is nil.
Let $a\in R$. Since $R/J(R)$ is tripotent, we see that $(a^2-a)-(a^2-a)^3, a^3-a\in N(R)$, and so $3a^2-3a\in N(R)$. This shows that $(-2a^2)^2-(-2a^2)=4a^4+2a^2
=(6a^4-2a^4)+2a^2=6a^4+2a(a-a^3)\in N(R)$.
Moreover, $(a+2a^2)^2-(a+2a^2)=a^2+4a^3+4a^4-a-2a^2=(3a^2-3a)+4(a^3-a)+6a+4a(a^3-a)\in N(R)$. In light of ~\cite[Lemma 3.5]{KW}, there exist $f(t),g(t)\in {\Bbb Z}[t]$ such that
$(-2a^2)-f(a), (a+2a^2)-g(a)\in N(R)$, $f(a)=f^2(a)$ and $g(a)=g^2(a)$.
Therefore $a-\big(f(a)+g(a)=\big((a+2a^2)-g(a)\big)+\big((-2a^2)-f(a)\big)\in N(R)$. Hence, $a=f(a)+g(a)+w$ with $w\in N(R)$.
One easily checks that $af(a)=f(a)a$ and $ag(a)=g(a)a$, and then
$f(a),g(a)$ and $w$ commute. Therefore $R$ is strongly 2-nil-clean, as asserted.\end{proof}

A ring $R$ a right (left) quasi-duo ring if every
maximal right (left) ideal of $R$ is an ideal. For instance, local
rings, duo rings and weakly right (left) duo rings are all right
(left) quasi-duo rings. Every abelian exchange ring is a right
(left) duo ring (cf. ~\cite{Yu}).

\begin{thm} A ring $R$ is strongly 2-nil-clean if and only if
\end{thm}
\begin{enumerate}
\item [(1)] {\it $R$ is 2-nil-clean;}
\vspace{-.5mm}
\item [(2)] {\it $R$ is right (left) quasi-duo;}
\vspace{-.5mm}
\item [(3)] {\it $J(R)$ is nil.}
\end{enumerate}
\begin{proof} $\Longrightarrow$ $(1)$ is obvious. By Lemma 2.7, $R/J(R)$ is tripotent and then it is commutative. Let $M$ be a right (left) maximal ideal of $R$. Then $M/J(R)$ is an ideal of $R/J(R)$. Let $x\in M, r\in R$. Then $\overline{rx}\in M/J(R)$, and then $rx\in M+J(R)\subseteq M$. This shows that $M$ is an ideal of $R$. Thus $R$ is right (left) quasi-duo. $(3)$ is follows from Lemma 2.7.

$\Longleftarrow$ As $R$ is 2-nil-clean, $R/J(R)$ is 2-nil-clean. Since $R$ right (left) is quasi-duo, then by ~\cite[Lemma 2.3]{Yu}, every nilpotent in $R$ contains in $J(R)$. Let $e\in R/J(R)$ be an idempotent. As $J(R)$ is nil, we can find an idempotent $f\in R$ such that $e=f+J(R)$.
For any $r\in R$, $fr(1-f)\in J(R)$, and then $e\overline{r}=e\overline{r}e$. Likewise, $\overline{r}e=e\overline{r}e$. Thus, $e\overline{r}=\overline{r}e$, i.e., $R/J(R)$ is abelian. Hence, $R/J(R)$ is tripotent, by Lemma 2.6. As $J(R)$ is nil, it follows by Lemma 2.7 that $R$ is strongly 2-nil-clean.\end{proof}

A natural problem is if the matrix ring over a strongly 2-nil-clean ring is strongly 2-nil-clean. The answer is negative as the following shows.

\begin{exam} Let $n\geq 2$. then matrix ring $M_n(R)$ is not strongly 2-nil-clean for any ring $R$.
\end{exam}
\begin{proof} Let $R$ be a ring, and let $A=\left(
\begin{array}{cc}
1_R&1_R\\
1_R&0
\end{array}
\right)$. Then $A^3-A=\left(
\begin{array}{cc}
2&1_R\\
1_R&1_R
\end{array}
\right)$. One checks that $\left(
\begin{array}{cc}
2&1_R\\
1_R&1_R
\end{array}
\right)^{-1}=\left(
\begin{array}{cc}
1_R&-1_R\\
-1_R&2
\end{array}
\right)$, and so $A^3-A$ is not nilpotent. If $M_n(R)$ is strongly 2-nil-clean, as in the proof of Lemma 2.7, $A^3-A$ is nilpotent, a contradiction, and we are done.\end{proof}

\section{2-Nil-clean Matrix Rings}

In ~\cite[Corollary 1]{H}, Han and Nicholson proved that every matrix ring of a clean ring (i.e., every element is the sum of an idempotent and a unit) is clean.
By using a similar route, we easily see that every matrix over a 2-nil-clean ring is the sum of two idempotent matrices and an invertible matrix.
As seen in Example 2.9, there exist some matrices over an arbitrary strongly 2-nil-clean ring which is not strongly 2-nil-clean. The purpose of this section is to investigate certain strongly 2-nil-clean rings over which every matrix is 2-nil-clean. We have

\begin{lem} $M_n({\Bbb Z}_3)$ is 2-nil-clean.\end{lem}
\begin{proof} As every
matrix over a field has a Frobenius normal form, and that 2-nil-clean matrix is invariant under the similarity, we may assume
that$$A=\left(
\begin{array}{cccccc}
0&&&&&c_0\\
1&0&&&&c_1\\
&1&0&&&c_2\\
&&&\ddots&&\vdots\\
&&&\ddots&0&c_{n-2}\\
&&&&1&c_{n-1}
\end{array}
\right).$$

Case I. $c_{n-1}=1$. Choose $$W=\left(
\begin{array}{cccccc}
0&&&&&0\\
1&0&&&&0\\
&1&0&&&0\\
&&&\ddots&&\vdots\\
&&&\ddots&0&0\\
&&&&1&0
\end{array}
\right), E=\left(
\begin{array}{cccccc}
0&&&&&c_0\\
0&0&&&&c_1\\
&0&0&&&c_2\\
&&&\ddots&&\vdots\\
&&&\ddots&0&c_{n-2}\\
&&&&0&1
\end{array}
\right).$$ Then $E^2=E$, and so $A=E+0+W$ is 2-nil-clean.

Case II. $c_{n-1}=-1$. Choose $$W=\left(
\begin{array}{cccccc}
0&&&&&0\\
1&0&&&&0\\
&1&0&&&0\\
&&&\ddots&&\vdots\\
&&&\ddots&0&0\\
&&&&1&0
\end{array}
\right), E=\left(
\begin{array}{cccccc}
0&&&&&c_0\\
0&0&&&&c_1\\
&0&0&&&c_2\\
&&&\ddots&&\vdots\\
&&&\ddots&0&c_{n-2}\\
&&&&0&-1
\end{array}
\right).$$ Then $E^2=-E$, and so $A=(I_2-E)+I_2+W$ is 2-nil-clean.

Case III. $c_{n-1}=0$.

If $n=2$, then $$\left(
\begin{array}{cc}
0&c_0\\
1&0
\end{array}
\right)=\left(
\begin{array}{cc}
1&0\\
0&1
\end{array}
\right)+\left(
\begin{array}{cc}
-1&1\\
1&-1
\end{array}
\right)+\left(
\begin{array}{cc}
0&c_0-1\\
0&0
\end{array}
\right)~\mbox{is 2-nil-clean}.$$

If $n=3$, then $$\left(
\begin{array}{ccc}
0&0&c_0\\
1&0&c_1\\
0&1&0
\end{array}
\right)=\left(
\begin{array}{ccc}
0&0&0\\
0&1&0\\
1&0&1
\end{array}
\right)+\left(
\begin{array}{ccc}
0&0&0\\
1&-1&1\\
-1&1&-1
\end{array}
\right)+\left(
\begin{array}{ccc}
0&0&c_0\\
0&0&c_1-1\\
0&0&0
\end{array}
\right)$$ is 2-nil-clean.

If $n\geq 4$, we have
$$\begin{array}{l}
A=\\
\left(
\begin{array}{ccccccc}
0&0&\cdots&0&0&0&0\\
0&0&\cdots&0&0&0&0\\
0&0&\cdots&0&0&0&0\\
\vdots&\vdots&\ddots&\vdots&\vdots&\vdots&\vdots\\
0&0&\cdots&0&0&0&0\\
0&0&\cdots&0&0&1&0\\
0&0&\cdots&0&1&0&1
\end{array}
\right)+
\left(
\begin{array}{ccccccc}
0&0&\cdots&0&0&0&0\\
0&0&\cdots&0&0&0&0\\
0&0&\cdots&0&0&0&0\\
\vdots&\vdots&\ddots&\vdots&\vdots&\vdots&\vdots\\
0&0&\cdots&0&0&0&0\\
0&0&\cdots&0&1&-1&1\\
0&0&\cdots&0&-1&1&-1
\end{array}
\right)\\
+\left(
\begin{array}{ccccccc}
0&0&\cdots&0&0&0&c_0\\
1&0&\cdots&0&0&0&c_1\\
0&1&\cdots&0&0&0&c_2\\
\vdots&\vdots&\ddots&\vdots&\vdots&\vdots&\vdots\\
0&0&\cdots&1&0&0&c_{n-3}\\
0&0&\cdots&0&0&0&c_{n-2}-1\\
0&0&\cdots&0&0&0&0
\end{array}
\right)
\end{array}$$ is the sum of two idempotents and a nilpotent. This implies that $A\in M_n({\Bbb Z}_3)$ is 2-nil-clean. Therefore $M_n({\Bbb Z}_3)$ is 2-nil-clean.\end{proof}

\begin{lem} Let $R$ be tripotent. Then $M_n(R)$ is 2-nil-clean for all $n\in {\Bbb N}$.\end{lem}
\begin{proof} Let $A\in M_n(R)$, and
let $S$ be the subring of $R$ generated by the entries of $A$.
That is, $S$ is formed by finite sums of monomials of the form:
$a_1a_2\cdots a_m$, where $a_1,\cdots,a_m$ are entries of $A$. Since $R$ is a
commutative ring in which $6=0$, $S$
is a finite ring in which $x=x^3$ for all $x\in S$. By virtue of Lemma 2.6,
$S$ is isomorphic to finite direct
product of ${\Bbb Z}_2$ and/or ${\Bbb Z}_3$. In terms of Lemma 3.1 and Example 2.2 (2), $M_n(S)$ is
2-nil-clean. As $A\in M_n(S)$, $A$ is the sum of two idempotent matrices and a nilpotent matrix over $S$,
as desired.\end{proof}

\begin{thm} Let $R$ be 2-primal. If $R$ is strongly 2-nil-clean, then $M_n(R)$ is 2-nil-clean for all $n\in {\Bbb N}$.\end{thm}
\begin{proof} Since $R$ is strongly 2-nil-clean, it follows by Lemma 2.7 that $J(R)$ is nil and $R/J(R)$ is tripotent.
In virtue of Lemma 3.2, $M_n(R/J(R))$ is 2-nil-clean. Furthermore,
$J(R)\subseteq N(R)=P(R)\subseteq J(R)$, we get $J(R)=P(R)$. Hence, $M_n(J(R))=M_n(P(R))=P(M_n(R))$ is nil.
Since $M_n\big(R/J(R)\big)\cong M_n(R)/M_n(J(R))$, it follows by Theorem 2.3 that $M_n(R)$ is 2-nil-clean. This completes the proof.\end{proof}

\begin{cor} Let $R$ be
a commutative 2-nil-clean ring. Then $M_n(R)$ is 2-nil-clean for all $n\in {\Bbb N}$.\end{cor}

\begin{cor} Let $R$ be a commutative weakly nil-clean ring. Then $M_n(R)$ is 2-nil-clean for all $n\in {\Bbb N}$.\end{cor}
\begin{proof} As every commutative weakly nil-clean ring is strongly 2-nil-clean 2-primal ring, we obtain the result, by Theorem 3.3.\end{proof}

\begin{exam} Let $m=2^k3^l (k,l\in {\Bbb N})$. Then $M_n ({\Bbb Z}_{m})$ is 2-nil-clean for all $n\in {\Bbb N}$.\end{exam}
\begin{proof} In light of ~\cite[Example 9]{BD}, ${\Bbb Z}_m$ is a commutative weakly nil-clean ring, hence the result by Corollary 3.5.\end{proof}

\begin{lem} (~\cite[Lemma 6.6]{KWZ}) Let $R$ be of bounded index. If $J(R)$ is nil, then $M_n(R)$ is nil for all $n\in {\Bbb N}$.
\end{lem}

\begin{thm} Let $R$ be of bounded index. If $R$ is strongly 2-nil-clean, then $M_n(R)$ is 2-nil-clean for all $n\in {\Bbb N}$.
\end{thm}
\begin{proof} By virtue of Lemma 3.7, $M_n(J(R))$ is nil. In view of Lemma 2.7, $R/J(R)$ is tripotent. Thus, $M_n(R/J(R))$ is 2-nil-clean, in terms of Lemma 3.2. Since $M_n(R/J(R))/J(M_n(R))\cong M_n(R/J(R))$, according to Theorem 2.3, $M_n(R)$ is 2-nil-clean.\end{proof}

\begin{cor} Let $R$ be a ring, and let $m\in {\Bbb N}$. If $(a-a^3)^m=0$ for all $a\in R$, then $M_n(R)$ is 2-nil-clean for all $n\in {\Bbb N}$.
\end{cor}
\begin{proof} Let $x\in J(R)$. Then $(x-x^3)^m=0$, and so $x^m=0$. This implies that $J(R)$ is nil. In light of ~\cite[Theorem A.1]{HTY}, $N(R)$ forms an ideal of $R$, and so $N(R)\subseteq J(R)$. Hence, $J(R)=N(R)$ is nil. Further, $R/J(R)$ is tripotent. In light of Lemma 2.7, $R$ is strongly 2-nil-clean. If $a^k=0 (k\in {\Bbb N}$, then $1-a,1+a\in U(R)$, and so $1-a^2=(1-a)(1+a)\in U(R)$.
By hypothesis, $a^m(1-a^2)^m=0$. Hence, $a^m=0$, and so $R$ is of bounded index. This complete the proof, by
Theorem 3.8.\end{proof}

A ring $R$ is a 2-Boolean ring provided that $a^2$ is an idempotent for all $a\in R$.

\begin{cor} Let $R$ be a 2-Boolean ring. Then $M_n(R)$ is 2-nil-clean for all $n\in {\Bbb N}$.\end{cor}
\begin{proof} Let $a\in R$. Then $a^2=a^4$. Hence, $a^2(1-a^2)=0$. This shows that
$(1-a^2)^2a^2(1-a^2)a=0$, i.e., $(a-a^3)^3=0$. In light of Corollary 3.9, the result follows.\end{proof}

Let $n\geq 2$ be a fixed integer. Following Tominaga and Yaqub, a
ring $R$ is said to be generalized n-like provided that for any
$a,b\in R$, $(ab)^n-ab^n-a^nb+ab=0$ (~\cite{TY}).

\begin{cor} Let $R$ be a generalized 3-like ring. Then $M_n(R)$ is 2-nil-clean for all $n\in {\Bbb N}$.
\end{cor}
\begin{proof} Let $a\in R$. Then $(a-a^3)^2=0$, hence the result by Corollary 3.9.\end{proof}

Recall that a ring $R$ is strongly SIT-ring if every element in $R$ is the sum of an idempotent and a tripotent that commute ~(cf. ~\cite{Y}). We have

\begin{cor} Let $R$ be a strongly SIT-ring. Then $M_n(R)$ is 2-nil-clean for all $n\in {\Bbb N}$.
\end{cor}
\begin{proof} Let $R$ be a strongly SIT-ring, and let $a\in R$. In view of~\cite[Theorem 3.10]{Y}, we see that $a^6=a^4$; hence, $a^4(1-a^2)=0$. This implies that $(a-a^3)^5=0$. In light of Corollary 3.9, we complete the proof.\end{proof}

\end{document}